\documentclass[reqno,11pt]{amsart}
\usepackage[utf8]{inputenc}
\DeclareUnicodeCharacter{0301}{}

\calclayout
\usepackage{amsthm,amsfonts,amstext,amssymb,mathrsfs,amsmath,latexsym,mathtools,amsaddr}
\usepackage{enumerate}
\usepackage{bm,accents}
\usepackage{mdwlist}
\usepackage[abs]{overpic}
\usepackage[mathscr]{euscript}
\usepackage{mathrsfs}
\DeclareMathAlphabet{\mathpzc}{OT1}{pzc}{m}{it}
\usepackage{parskip}
\usepackage{hyperref}
\usepackage{comment}
\usepackage{color}

\usepackage{cleveref}

\crefname{equation}{equation}{equations}

\theoremstyle{plain}
\newtheorem{thm}{Theorem}[section]
 
\newtheorem{lm}[thm]{Lemma}
\newtheorem{cor}[thm]{Corollary}

\theoremstyle{definition}
\newtheorem{definition}[thm]{Definition}

\newtheorem{method}[thm]{Method}

\theoremstyle{remark}
\newtheorem{remark}[thm]{Remark}

\numberwithin{equation}{section}

\newcommand{\hypergeometric}[5]{\, _{#1}F_{#2}\left(\left.\begin{array}{c}#3\\#4\end{array}\right|#5\right)}

\title{An electrostatic model for the roots of discrete classical orthogonal polynomials}

\author{Joaquín S\'anchez-Lara}

\address{Departamento de Matem\'atica Aplicada, Universidad de Granada, Spain}

\email{jslara@ugr.es}

\date{\today}

\keywords{Classical Orthogonal polynomials;  Askey Tableau; Zeros; Electrostatic model; Equilibrium; Linear difference equations}

\subjclass[2020]{Primary:  42C05; Secondary: 30C15; 31A15; 33C45; 33C47}

\begin{document}

\begin{abstract}

An electrostatic model is presented to describe the behaviour of the roots of classical discrete orthogonal polynomials. Indeed, this model applies in the more general frame of  polynomial solutions of second-order linear difference equations
$$A\Delta_h\nabla_h y+B\Delta_h y+ C y=0\,,$$
where $A$, $B$ and $C$ are polynomials and 
$$\Delta_h f(x)=f(x+h)-f(x)\qquad \text{ and }\qquad \nabla_h f(x)=f(x)-f(x-h)$$
 with $h>0$.



\end{abstract}

\maketitle

\section{Introduction}

Given a measure $\mu$ (orthogonality measure) supported on $S\subset\mathbb{C}$, the sequence of orthogonal polynomials (OP) with respect to $\mu$ is a family of polynomials $\{p_n\}$ with $p_n$ of degree $n$ and satisfying $$\int_S x^j p_n(x)d\mu(x)\begin{cases}=0\,,& j=0,1,\dots,n-1\,,\\\neq 0\,,& j=n\,.\end{cases}$$
When $\mu$ is positive, $S\subset\mathbb{R}$ and $n$ is less than the cardinal of $S$, the existence of such $p_n$ is guarantied and its roots are all real, simple and are located on the convex hull of $S$ (if $S$ is an interval, the roots are indeed in its interior). 


One of the topics that has aroused considerable interest since the beginning of the theory of orthogonal polynomials is the analysis of their zeros, and a very elegant yet useful way to describe them was given by T. Stieltjes through the following electrostatic model \cite{Stieltjes1885c,Stieltjes1885a}. If a polynomial $p$ of degree $n$ with real roots $x_1<x_2<\dots<x_n$ satisfies a second-order linear differential equation with polynomial coefficients
\begin{equation}\label{ODEgen}
Ay''+By'+Cy=0\,,
\end{equation}
then its roots coincide with the following distribution of points. Consider $n$ unitary charges that can move freely along some continuum $\Sigma$ of the real line. A charge  at $x\in\mathbb{R}$ repels another charge placed at $y\in\mathbb{R}\setminus\{x\}$ with a force given by
\begin{equation}\label{ForceLog}
F_{x}^{\text{log}}(y):=\frac{1}{y-x}\,.
\end{equation}
Additionally there is an external field, $\phi$, which exerts an extra force over a charge located at $x$ given by
$$-\phi'(x)=\frac{B(x)}{2A(x)}\,.$$
In this scenario may appear a \emph{critical distribution}, i.e., a distribution such that the total force at each charge vanishes (the use of the adjective \emph{critical} will be clearer forward). This critical distribution is precisely the one given by the roots $x_1,\dots x_n$ of the polynomial solution of \ref{ODEgen}. Stieltjes is considered the father of this interpretation, although he used constrained in the case of an unbounded $\Sigma$ instead of external fields.

Let us see a brief sketch of the arguments that connect the differential equation \eqref{ODEgen} with this electrostatic interpretation. Evaluating \eqref{ODEgen} with $y=p$ at $x_j$ we find that
$$\frac{p''(x_j)}{2p'(x_j)}+\frac{B(x_j)}{2A(x_j)}=0\,,\qquad j=1,\dots n\,,$$
or equivalently
\begin{equation}\label{Eqgen}
\sum_{k=1,k\neq j}^n\frac{1}{x_j-x_k}+\frac{B(x_j)}{2A(x_j)}=0\,,\qquad j=1,\dots n\,.
\end{equation}
If we put a unitary charge at each $x_j$, $j=1,\dots, n$, then the sum in $k$ can be interpreted as the sum of the forces that the charges at $x_1,\dots,x_{j-1},x_{j+1},\dots, x_n$ exert over the charge at $x_j$. The term with the quotient $B/(2A)$ is interpreted as the force that the external field exerts over the charge $x_j$, so $\phi$ is a primitive of $-B/(2A)$. Then, identity \eqref{Eqgen} says that the total force at any $x_j$, $j=1,\dots, n$, vanishes, so this distribution is a critical distribution.

In the literature, this electrostatic model is usually expressed in terms of potentials and energies rather than forces. The energy of a system of unitary charges at $x_1<\dots<x_n$ under an external field $\phi$ is defined as
$$\mathcal{E}_\phi^{\text{log}}(x_1,\dots,x_n)=\sum_{j,k=1,j\neq k}^nV_{x_k}^{\text{log}}(x_j)+2\sum_{j=1}^n\phi(x_j)\,,$$
where $V_y^{\text{log}}$ is the potential created by a charge and is given by
$$V_y^{\text{log}}(x)=\log\frac{1}{|y-x|}\,,\qquad x\in\mathbb{R}\setminus\{x_j\}\,.$$
Due to their nature, these potentials are called \emph{logarithmic potentials}. Typically, the problem consists in, with a fixed $n$, minimizing the energy $\mathcal{E}_\phi^{\text{log}}$ among all the distributions of $n$ points in the continuum $\Sigma$ of the real line, that is, to find the \emph{equilibrium distribution}. Observe that a critical point for the energy $\mathcal{E}_\phi^{\text{log}}$ is determined by equations \eqref{Eqgen} (this is the reason for the use of the term \emph{critical} for the critical distributions).
%
%

Stieltjes was the first to apply these ideas to orthogonal polynomials. He applied them to classical orthogonal polynomials (see \cite{Stieltjes1885c,Stieltjes1885a}) obtaining the electrostatic interpretation of their roots (see table \ref{TableClas}). In particular, this interpretation allows us to intuit (and sometimes to prove) the behaviour of these zeros as we move some parameters. For instance, for Jacobi polynomials $P_n^{(\alpha,\beta)}$, if coefficient $\alpha$ increases, the repulsive charge located at $1$ increases its size, so the roots must move to the left. Another fact that can be intuited is the asymptotic as $n\to\infty$. The repulsive charges at $\pm 1$ become negligible as $n\to\infty$ since the force that they exerts over a charge at a root of $P_n^{(\alpha,\beta)}$ cannot compare to the force that the other $n-1$ charges at the rest of the roots of $P_n^{(\alpha,\beta)}$ exert. Then, once we normalize the total mass of the system of charges, the limit distribution must be a (continuous) distribution which remains in equilibrium in $[-1,1]$ by itself (i.e., without an external field), that is, the equilibrium measure in $[-1,1$]
$$\frac{d\nu(x)}{dx}=\frac{1}{\pi\sqrt{1-x^2}}\,.$$

 \renewcommand{\arraystretch}{1.8}
\begin{table}\label{TableClas}
\begin{center}
\begin{tabular}{|l|c|c|c|c|}
\hline
Family of OP& $A$ & $B$ & Support& External force\\[2mm]
\hline
Jacobi $P_n^{(\alpha,\beta)}$ & $x^2-1$ & $(\alpha+\beta+2)x+\alpha+\beta$ & $[-1,1]$ & $\displaystyle\frac{\alpha+1}{2(x-1)}+\frac{\beta+1}{2(x+1)}$\\[2mm]
\hline
Laguerre $L_n^{(\alpha)}$ & $x$ & $-x+\alpha+1$ & $[0,+\infty)$ & $\displaystyle\frac{\alpha+1}{2x}-\frac{1}{2}$\\[2mm]
\hline
Hermite $H_n$ & $1$ & $-2x$ & $\mathbb{R}$ & $-x$\\[2mm]
\hline
\end{tabular}
\caption{For the families of classical orthogonal polynomials: coefficients of the differential equation \eqref{ODEgen}, support of the orthogonality measure (which coincides with $\Sigma$) and the force that the external field produces.}
\end{center}
\end{table}

It was not until the end of the 20th century that this electrostatic interpretation was extended. It was M. E. H. Ismail (see \cite{Ismail2000a,Ismail2000b}) who found a way to obtain a differential equation like \eqref{ODEgen} for general orthogonal polynomials when the orthogonality measure $\mu$ is absolutely continuous (and some additional conditions on $\mu'$). From this differential equation, the electrostatic model follows from the previous commented ideas. Recently a different approach has been studied for orthogonal polynomials with respect to semiclassical orthogonality measures (measures such that the derivative of $\log(\mu')$ is a rational function) (see \cite{SanchezLara2023} for quasi orthogonal polynomials and type II multiple orthogonal polynomials and \cite{Bertola24} for degenerate orthogonal polynomials) but it also applies only to the case of absolutely continuous orthogonalities.

 Obtaining an electrostatic interpretation of the zeros of orthogonal polynomials with respect to a discrete measure has been an open problem until the work of Rutka and Smarzewski \cite{Smarzewski2023}.  There, the authors consider classical orthogonal polynomials with respect to a positive discrete measure supported in $\Sigma=\{\xi_0,\xi_1,\dots \xi_{N}\}$ with $N=\mathbb{N}\cup\{+\infty\}$ and $\xi_j=\xi_0+hj$ with $h>0$. The weight $\rho$ associated with $\mu$ satisfies a Pearson equation
 $$\Delta_h (A\rho)=B\rho\,,\qquad \left.A\rho x^j\right|_{\xi_0,\xi_N}=0\,,$$
 where $A$ and $B$ are polynomials of degree at most $\deg(A)\leq 2$ and $\deg (B)\leq 1$ respectively. Then, they define the functional
$$G(z_1,\dots,z_n;h)=\prod_{j,k=1,j\neq k}^n\left((z_j-z_k)^2-h^2\right)\,\prod_{j=1}^nA(z_j)\rho(z_j)\,,$$
which plays a similar role as $-\log(\mathcal{E}_\phi^{\log})$, and look for maximizing $G$ (minimizing the energy). They prove some difference properties about $G$ and the roots of the orthogonal polynomials, which shows that, although these roots may not correspond to the maximizing distribution for $G$, they are close to it.

In this paper we present a different approach that provides an electrostatic model for the roots of orthogonal polynomials that present differential properties through difference operators $\Delta_h$ and $\nabla_h$. Although some ideas are strongly influenced by \cite{Smarzewski2023}, the obtained results are not equivalent to those of \cite{Smarzewski2023} since the energies of both models are essentially different.

The paper is structured as follows. In Section \ref{Sec2} we introduce the electrostatic model in terms of forces as well as in terms of energies. There, we connect the polynomial solutions of a second-order linear difference equation with critical distributions of points (see Theorem \ref{TeoEDElechPos}). We will see that under some rather general conditions, the existence and uniqueness of the equilibrium distribution (minimizer of the energy) is guarantied (see Theorem \ref{TeoMinEnergy}). Moreover, we present a method (see Method \ref{Alg}) to obtain this equilibrium distribution, which in addition explains some properties of the zeros. In section \ref{Sec3} we apply these results to orthogonal polynomials in the Askey scheme which satisfy differential properties in terms of $\Delta_1$ and $\nabla_1$ and, as an example of the applicability of the results in section \ref{Sec2}, we obtain the monotonicity of the roots with respect to some of the parameters of these polynomials.

\section{Electrostatic model}\label{Sec2}



Consider the following electrostatic problem with $h>0$ fixed. We put $n$ free charges in some interval $[a,b]$ with $a,b\in\mathbb{R}\cup\{\pm\infty\}$. The force that a charge placed at $x$ exerts over another charge at $y$ is given by
\begin{equation}\label{Forceh}
F_x(y):=\frac{1}{2h}\log\frac{y-x+h}{y-x-h}\,,\qquad y\in \mathbb{R}\setminus[x-h,x+h]\,.
\end{equation}
We observe that there is a region $[x-h,x+h]$ around $x$ where we have not defined it, this is a forbidden region to any other charge (where the quotient inside the logarithm is negative). We shall call \emph{exclusion radius} to the parameter $h$. Apart from this zone, this force exhibits the expected properties that one can imagine:
\begin{itemize}
\item For $y>x+h$ is positive (repulsive) and its effect decreases as $y$ moves away from $x$. For $y\searrow x+h$ diverges to $+\infty$ and for $y\to+\infty$ converges to $0$.
\item For $y<x-h$ is negative (repulsive) and its effect decreases as $y$ moves away from $x$. For $y\nearrow x-h$ diverges to $-\infty$ and for $y\to-\infty$ converges to $0$.
\item It presents odd symmetry with respect to $x$
$$F_x(x+y)=-F_x(x-y)\,,\qquad \text{for }y>h\,.$$
\item It is invariant under translations, $F_{x+c}(y+c)=F_x(y)$ for any $c\in\mathbb{R}$.
\item The force that a charge at $x$ exerts on another charge at $y$ is the opposite of the force that the charge at $y$ exerts on the charge at $x$
    $$F_x(y)=-F_y(x)\,.$$
\end{itemize}
Moreover, in figure \ref{FigCamposFuerzas} we can see that this force is not so far from the force that the logarithmic potentials induce, $F^{\log}$. This similarity increases as the distance between $x$ and $y$ increases because of the known approximation
$$F_x(y)=\frac{1}{2h}\log\left(1+\frac{2h}{y-x-h}\right)= \frac{1}{y-x+h}+O\left(\frac{1}{(x-y-h)^2}\right)=F_x^{\log}(y)+O\left(\frac{1}{(x-y)^2}\right)\,,$$
as $y-x\to \pm \infty$. For instance, we can see that, with $h=1$, at a distance of 2 units from the charge, $F_0$ is approximately 10\% greater than $F_0^{\log}$ and at a distance of $3$ units, $F_0$ is only approximately 4\% greater than $F_0^{\log}$. Observe also that this force converges to that of logarithmic potentials when $h\searrow 0$
$$\lim_{h\searrow 0}F_x(y)=\frac{1}{y-x}=F_x^{\log}(y)\,.$$
\begin{figure}
\begin{center}
\includegraphics[width=7.5cm]{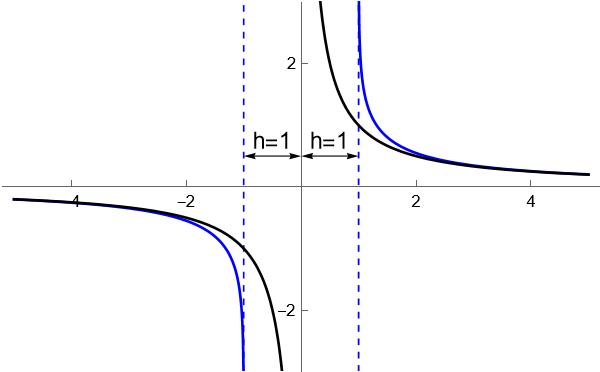}\hspace{1cm}
\includegraphics[width=7.5cm]{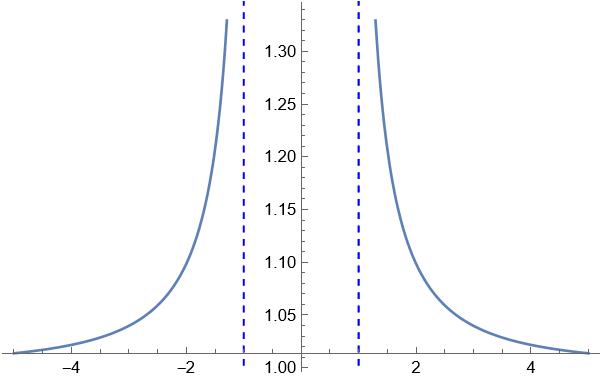}
\caption{Left: Force field given by $F_0$ with $h=1$ defined by \eqref{Forceh} in blue and force field $F_0^{\log}$ defined by \eqref{ForceLog} in black. Right: Rate between the two force fields. }\label{FigCamposFuerzas}
\end{center}
\end{figure}

\begin{remark} This charges has been considered unitary, if the size of the charge at $x$ is $\lambda$, then the force that it exerts is given by $\lambda F_x$. However, in this context, the exclusion radius, $h$, also influences on the force, as $h$ increases the absolute value of the force also increases. In section \ref{Sec3} we will make use of these parameters (size and exclusion radius) in order to describe some external fields.
\end{remark}

Additionally, we assume that these charges are under the effect of an external field $\phi$ which produces a force  given by
\begin{equation}\label{ExtForcehPos}
F_{\phi}(x):=-\phi'(x)=\frac{1}{2h}\log\left(1+\frac{B(x)}{A(x)}\right)\,,
\end{equation}
where $A$ and $B$ are two relative prime polynomials such that $1+B/A>0$ in $(a,b)$ (observe that $\phi$ can be defined not only in the open set $(a,b)$ but also at the endpoints $a$ and $b$ assuming that they are not $\pm\infty$).

\begin{definition}\label{DefCrDis}
A \emph{critical distribution} under the external field $\phi$ in $(a,b)$ is a distribution $a< x_1<x_2<\dots<x_n< b$ with $x_{j+1}>x_j+h$ for $j=1,\dots,n-1$, such that
\begin{equation}\label{TotalForcehPos}
\sum_{k=1,k\neq j}^nF_{x_k}(x_j)+F_{\phi}(x_j)=0\,,\qquad j=1,\dots,n.
\end{equation}
\end{definition}

The problem of finding critical distributions can also be expressed in terms of energies and potentials. The potential associated with a charge at $x$ is given by
\begin{equation}\label{Pothpos}
V_x(y)=-\int F_x(y)dy=\frac{x-y}{2h}\log\left(\frac{y-x+h}{y-x-h}\right)-\frac{1}{2}\log\big((y-x+h)(y-x-h)\big)\,,
\end{equation}
(see figure \ref{FigPotenciales})
which can be defined for $y\in\mathbb{R}\setminus(x-h,x+h)$, and
the energy of the system is
\begin{equation}\label{EnergiahPos}
\mathcal{E}_{\phi}(x_1,\dots,x_n)=\sum_{j,k=1,j\neq k}^nV_{x_k}(x_j)+2\sum_{j=1}^n\phi(x_j)\,.
\end{equation}
A critical distribution is just a critical point of $\mathcal{E}_\phi$. In addition, we use the following definition.
\begin{definition}
  An \emph{equilibrium distribution} under the external field $\phi$ in $(a,b)$ is a distribution $a\leq x_1<x_2<\dots<x_n\leq b$ with $x_{j+1}>x_j+h$ for $j=1,\dots,n-1$, such that it minimizes the energy $\mathcal{E}_\phi$ among all distributions of $n$ points in $(a,b)$ with mutual distances greater than $h$.
\end{definition}

\begin{figure}
\begin{center}
\includegraphics[width=7.5cm]{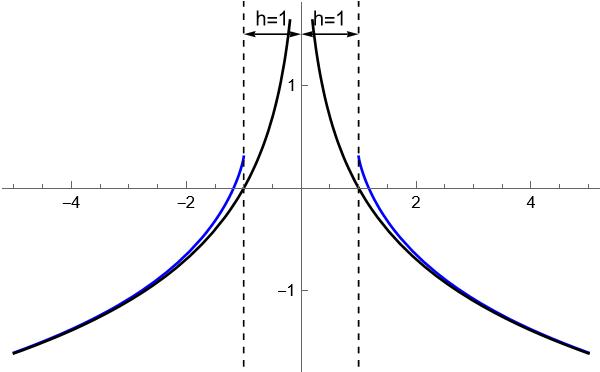}
\caption{Potentials $V_0^{\text{log}}$ in black and $V_0$ with $h=1$ in blue that an unitary charge at $x=0$ creates.}\label{FigPotenciales}
\end{center}
\end{figure}

Observe that if the interval $(a,b)$ is bounded, then such distributions consist of at most $\lceil(b-a)/h\rceil$ points where $\lceil x\rceil$ is the ceiling function defined as the smallest integer greater than or equal to $x$. In addition, if $b-a=Nh$ for some $N\in\mathbb{N}$, then the distribution $\{a+j h:j=0,\dots, N\}$ is the unique distribution of $N+1$ points such that their mutual distances are greater than or equal to $h$; of course, the forces at each charge diverge (so it has no sense from the point of view of forces), but the energy of such a system is a real number.


The following theorem establishes the relation between critical distributions and polynomial solutions of difference equations in terms of the $\Delta_h$ operator.

\begin{thm}\label{TeoEDElechPos} Consider two polynomials $A$ and $B$ and an interval $(a,b)$ in which $1+B/A>0$.
\begin{itemize}
\item[(i)]
If a distribution of points $a<x_1<x_2<\dots<x_n<b$ is a critical distribution under the external field \eqref{ExtForcehPos}, then the polynomial $p(x)=(x-x_1)\dots(x-x_n)$ is a solution of the difference equation
\begin{equation}\label{ED}
A\Delta_h\nabla_h y+B \Delta_h y+Cy=0\,,
\end{equation}
for some polynomial $C$ of degree $\max(\deg(A)-2,\deg(B)-1)$ at most.
\item[(ii)] Conversely, if a polynomial $p$ is a solution of a difference equation \eqref{ED} and its roots are simple, are in $(a,b)$ and their mutual distances are greater than $h$, then these roots correspond with a critical distribution.
\end{itemize}
\end{thm}
\begin{proof}
Let us see (i). Equations \eqref{TotalForcehPos} are equivalent to
$$-\sum_{k=1,k\neq j}^n\frac{1}{2h}\log\left(\frac{x_j-x_k+h}{x_j-x_k-h}\right)=\frac{1}{2h}\log\left(1+\frac{B(x_j)}{A(x_j)}\right)\,,
\qquad j=1,\dots n\,.$$
With some straightforward computations we deduce the following
$$\prod_{k=1,k\neq j}^n\frac{x_j-x_k-h}{x_j-x_k+h}=1+\frac{B(x_j)}{A(x_j)}\,,\qquad j=1,\dots n\,,$$
adding the factor for $k=j$
$$-\frac{p(x_j-h)}{p(x_j+h)}=1+\frac{B(x_j)}{A(x_j)}\,,\qquad j=1,\dots n\,.$$
Multiplying by $A(x_j)p(x_j+h)$ and adding some summands $p(x_j)=0$ in the correct places, it yields
$$A(x_j)\Big(p(x_j+h)-2p(x_j)+p(x_j-h)\Big)+B(x_j)\Big(p(x_j+h)-p(x_j)\Big)=0\,,\qquad j=1,\dots,n\,,$$
that is,
\begin{equation}\label{pol}
A\Delta_h\nabla_h p+B\Delta_h p\,,
\end{equation}
vanishes at the roots of $p$. Since \eqref{pol} is a polynomial of degree $n+\max(\deg(A)-2,\deg(B)-1)$, we have proven that \eqref{pol} can be expressed as $-Cp$  with $C$ a polynomial of degree $\max(\deg(A)-2,\deg(B)-1)$, and then, \eqref{ED} follows.

The proof of (ii) follows by going backward in the previous identities.
\end{proof}

Our next goal is to provide conditions that guarantee the existence of equilibrium distributions and the properties that they could satisfy. The obtained results are expressed in terms of equilibrium distributions, but they also have consequences on the roots of polynomial solutions of difference equations; such consequences will be discussed later. The following construction is useful for our purposes.


\begin{definition}
We say that an external field $\phi$ in an interval $(a,b)$, possibly unbounded, defined by \eqref{ExtForcehPos} is \emph{G-convex} if  $1+B/A$ is positive and decreasing in $(a,b)$ (hence $\phi$ is a convex function) and:
 \begin{itemize}
 \item if $a\in\mathbb{R}$ then $A(a)=0$, but if $a=-\infty$ then $\lim_{x\to-\infty}(1+B/A)(x)\in(1,+\infty]$;
 \item if $b\in\mathbb{R}$ then $(A+B)(b)=0$, but if $b=+\infty$ then $\lim_{x\to+\infty}(1+B/A)(x)\in[0,1)$.
 \end{itemize}
\end{definition}

\begin{remark}\label{RemarkCampoExterno}
  Note that due to the definition using polynomials $A$ and $B$, the possibility $(a,b)=\mathbb{R}$ is not possible. If we were to consider more general types of external fields this possibility could be included; however, the ones we are considering are sufficient for our purposes. Observe also that if $b=+\infty$ or $a=-\infty$, then the behaviour of $\phi$ as $|x|\to+\infty$ is at least linear, i.e., there exist constants $c,d>0$ such that
  $$\phi(x)>c |x|\,,\qquad \text{ for any }x\in (a,b) \text{ such that }|x|>d\,.$$
\end{remark}

If $\phi$ is G-convex, then for any distribution of charges $x_1<\dots <x_m$ in $(a,b)$, the function
$$F_{x_1,\dots,x_m}(x):=\sum_{j=1}^m F_{x_j}(x)+F_\phi(x)\,,$$
is decreasing in each interval $(x_j+h,x_{j+1}-h)$ and also in $(a,x_1-h)$ and $(x_m+h,b)$ (if such intervals are not empty). Furthermore, the limits of $F_{x_1,\dots,x_m}$ at the endpoints of such intervals are $\pm\infty$, which implies that, in each of these intervals, there exists a unique point $x_j^*$ such that $F_{x_1,\dots,x_m}(x_j^*)=0$. Hence, the following functions are  well defined
\begin{align*}
&R_0(x_1,\dots,x_m)=x_0^* \text{ such that }F_{x_1,\dots,x_m}(x_0^*)=0 \text{ and }x_0^*\in(a,x_1-h)\,, 
\\
&R_j(x_1,\dots,x_m)=x_j^* \text{ such that }F_{x_1,\dots,x_m}(x_j^*)=0 \text{ and }x_j^*\in(x_j+h,x_{j+1}-h)\,, \qquad j=1,\dots,m-1
\\
&R_m(x_1,\dots,x_m)=x_m^* \text{ such that }F_{x_1,\dots,x_m}(x_m^*)=0 \text{ and }x_m^*\in(x_m+h,b)\,, 
\end{align*}
in any case where such intervals are not empty. Observe that the functions $R_j$, $j=1,\dots,m-1$, have an electrostatic interpretation: fixed unitary charges at $x_1<\dots<x_m$ in $(a,b)$,  $R_j(x_1,\dots,x_m)$ corresponds with the point in the interval $(x_j+h,x_{j+1}-h)$ where to put a free charge in such a way that it remains in equilibrium. Equivalently $R_j(x_1,\dots,x_m)$ gives the solution of the  problem consisting of finding the place to set a charge in $(x_j+h,x_{j+1}-h)$
with the minimum energy of the system (with fixed charges at $x_1,\dots,x_m$). The same interpretations hold for $R_0(x_1,\dots,x_m)$ in $(a,x_1+h)$ and for $R_m(x_1,\dots,x_m)$ in $(x_m+h,b)$.  Now we can set out a procedure to find the equilibrium distribution.

\begin{method}\textbf{Iterating Energy Decreasing Method.}\label{Alg}
   Consider $\phi$ a G-convex external field in $(a,b)$ and $a<x_1<x_2<\dots<x_n<b$ with $x_{j+1}-x_j>h$, $j=1,\dots,n-1$. Define
   $$x_j^{(0)}=x_j\,,\qquad j=1,\dots,n\,,$$
   and
   $$x_j^{(\ell+1)}=R_j(x_1^{(\ell+1)},\dots,x_{j-1}^{(\ell+1)},x_{j+1}^{(\ell)},\dots,x_n^{(\ell)})\,,\qquad j=1,\dots,n\,,\quad \ell\in\mathbb{N}\,.$$
\end{method}
We explain what this method makes. The starting point is a distribution $x_1<\dots<x_n$. In the first step, we release the charge at $x_1$ (the others remain fixed), which moves to the equilibrium position $R_0(x_2,x_3,\dots,x_n)$; this can be done because the interval where $x_1$ is free is $(a,x_2-h)$, which is not empty. In the second step, it is the turn for $x_2$, we set only it free and it moves to its equilibrium position. These steps are repeated for $x_3$, then for $x_4$ and so on until $x_n$. Once this first iteration has been made for all charges, it begins again with $x_1$, then with $x_2$ and so on. Observe that the property $x_{j+1}^{(\ell)}>x_{j}^{(\ell)}+h$ holds for any $j=1,\dots, n$ and $\ell\in\mathbb{N}$. Observe also that in each step, the energy of the system $\mathcal{E}_{\phi}(x_1,\dots,x_n)$ decreases since the free charge at $x_j$ moves to the minimum of the problem
$$\min_{x_j\in(x_{j-1}+h,x_{j+1}-h)}\mathcal{E}_\phi(x_1,x_2,\dots,x_n)\,.$$
Observe that, in spite of the obvious dissimilarities, this method is in essence a Gauss-Seidel type method.

\begin{thm}\label{TeoMinEnergy}
  If $\phi$ is a G-convex external field in $(a,b)$, possibly unbounded, $h>0$ and $n\leq\lceil(b-a)/h\rceil$, then there exists a unique critical distribution of $n$ charges, which is indeed an equilibrium distribution. This distribution can be obtained by using method \ref{Alg} with any input data $x_1^{(0)}<\dots<x_n^{(0)}$ such that $x_{j+1}^{(0)}>x_j^{(0)}+h$ for $j=1,\dots,n-1$.
%
\end{thm}
\begin{proof} We prove it first for the case of a bounded interval $(a,b)$. Due to the condition $n\leq\lceil(b-a)/h\rceil$ we can choose a distribution of points $a<x_1^{(0)}<\dots<x_n^{(0)}<b$ such that $x_{j+1}^{(0)}-x_j^{(0)}>h$ for any $j=1,\dots, n-1$ and apply the method \ref{Alg}. While the method is running, all distributions are in the interior of the compact set
$$K=\{(y_1,\dots,y_n)\in\mathbb{R}^n: a\leq y_1\,, y_n\leq b\,,\text{ and } y_{j+1}-y_j\geq h \text{ for }j=1,\dots,n-1\}\,.$$
Consider the vector sequence $\{(X_1^{(\ell)},X_2^{(\ell)},\dots,X_n^{(\ell)})\}_{\ell\in\mathbb{N}}$
where
\begin{align*}
&X_1^{(\ell)}=(x_1^{(\ell+1)},x_2^{(\ell)},x_3^{(\ell)}\dots,x_{n-1}^{(\ell)},x_n^{(\ell)})\,,\\
&X_2^{(\ell)}=(x_1^{(\ell+1)},x_2^{(\ell+1)},x_3^{(\ell)}\dots,x_{n-1}^{(\ell)},x_n^{(\ell)})\,,\\
&\hspace{2.8cm}\vdots\\
&X_{n-1}^{(\ell)}=(x_1^{(\ell+1)},x_2^{(\ell+1)},x_2^{(\ell+1)},\dots,x_{n-1}^{(\ell+1)},x_n^{(\ell)})\,,\\
&X_{n}^{(\ell)}=(x_1^{(\ell+1)},x_2^{(\ell+1)},x_2^{(\ell+1)},\dots,x_{n-1}^{(\ell+1)},x_n^{(\ell+1)})\,.
\end{align*}
Because this sequence is in the compact set $K^n$, there is an infinite subset $L\in\mathbb{N}$ such that the subsequence $\{(X_1^{(\ell)},X_2^{(\ell)},\dots,X_n^{(\ell)})\}_{\ell\in L}$ converges to a $(X_1,X_2,\dots,X_n)$ with
\begin{align*}
&X_1=(x_1,x_2^*,x_3^*,\dots, x_{n-1}^*,x_n^*)\in K\,,\\
&X_2=(x_1,x_2,x_3^*,\dots, x_{n-1}^*,x_n^*)\in K\,,\\
&\hspace{2cm}\\
&X_{n-1}=(x_1,x_2,x_3,\dots, x_{n-1},x_n^*)\in K\,,\\
&X_n=(x_1,x_2,x_3,\dots, x_{n-1},x_n)\in K\,,
\end{align*}
for some $x_1,\dots, x_n$ and $x_2^*,\dots, x_n^*$.
Observe that the energy associated with $X_j^{(\ell)}$ is decreasing in $j$ and in $\ell$, so the energies associated with all the distributions $X_1$,\dots, $X_n$ are the same. This fact has several consequences. The first is that $x_j=x_j^*$ for $j=1,\dots n$. This is due again to the way the method works, in the step that goes from $X_1$ to $X_2$, the charge at $x_2^*$ is set free, while all the others remain fixed; then, it moves to the unique equilibrium position $x_2$, which in addition minimizes the energy; however since the energy of the distributions associated with $X_1$ and $X_2$ equals, then $x_2^*=x_2$. The same argument can be repeated for the rest of the charges, so $X_1=X_2=\dots=X_n=:X$.

Consider the function $f:{\rm Int}(K^n)\longrightarrow \mathbb{R}^n$ given by $f(Y_1,\dots,Y_n)=(f_1(Y_1),f_2(Y_2),\dots, f_n(Y_n))$ and
$$f_j(Y)=\sum_{k=1,k\neq j}^nF_{y_k}(y_j)+F_\phi(y_j)\,,\qquad\text{with }Y=(y_1,\dots,y_n) \,.$$
 Because $f$ is continuous and $f(X_1^{(\ell)},\dots, X_n^{(\ell)})=(0,\dots,0)$ we have got to either $X_1$ is in the border of $K$ or $f(X,\dots,X)=(0,\dots,0)$ and so $X$ is a critical distribution. Let us see that the first option is not possible. If there were some $x_{j+1}=x_j+h$ with $j\in\{1,\dots,n-1\}$ then in the sum
$$f_{j+1}(X_{j+1}^{(\ell)})=\sum_{k=1}^jF_{x_k^{(\ell+1)}}(x_{j+1}^{(\ell+1)})
+\sum_{k=j+2}^nF_{x_k^{(\ell)}}(x_{j+1}^{(\ell+1)})+F_\phi(y_j)=0\,,
$$
the addend $F_{x_j^{(\ell+1)}}(x_{j+1}^{(\ell+1)})$ would diverge to $+\infty$, so it would be compensated with the addend $F_{x_{j+2}^{(\ell)}}(x_{j+1}^{(\ell+1)})$ diverging to $-\infty$ (this is the unique that could do it) and then $x_{j+1}=x_j+h$ for $j=1,\dots,n-1$. Moreover, using the same argument we could prove that $x_1=a$ and $x_n=b$, which is not possible by the hypothesis $n<\lceil(b-a)/h\rceil$. Hence, we have prove that $X$ is in the interior of $K$ and  is a critical distribution.

The function which provides the energy, $\mathcal{E}_\phi$, for any distribution in $K$ is strictly convex because
$$\frac{\partial^2 \mathcal{E}_\phi(y_1,\dots,y_n)}{\partial y_j^2}
=\frac{1}{2h}\sum_{\ell=1,\ell\neq j}^n\frac{1}{(y_j-y_\ell)^2-h^2}+\phi''(y_j)\,,\qquad j=1,\dots, n\,,$$
where all addends are strictly positive, and for $j\neq \ell$
$$\frac{\partial^2 \mathcal{E}_\phi(y_1,\dots,y_n)}{\partial y_j\partial y_\ell}=\frac{1}{2h}\frac{1}{(y_j-y_\ell)^2-h^2}\,,$$
so we have that the centre of the Gershgorin circles are strictly positive and their radii are less than the centres. Then, the eigenvalues of the hessian matrix associated with $\mathcal{E}_\phi$ are strictly positive, which proves the convexity of the energy function. This implies that $X$ corresponds to the unique equilibrium distribution and that the sequence $(X_1^{(\ell)},\dots,X_n^{(\ell)})$  converges with
$$\lim_{\ell\to\infty}x_j^{(\ell)}= x_j\,.$$

Let us now see now the proof when the interval $(a,b)$ is not bounded, for instance if $b=+\infty$. Let us suppose that, to obtain a contradiction, $\max\{x_j^{(k)}:j=1,\dots n\}$ is not bounded. If this occurs, then there exists a subset $N\subset\mathbb{N}$ with infinity elements such that $\{x_n^{(k)}\}_{k\in N}$ diverges to $+\infty$ and $a\leq x_{j}^{(k)}\leq x_{n}^{(k)}$ for any $j=1,\dots, n-1$, then by remark \ref{RemarkCampoExterno}, the addend of $\mathcal{E}_\phi(x_1^{(k)},\dots,x_n^{(k)})$
$$2\phi(x_n^{(k)})\,,\qquad k\in N\,,$$
diverges to $+\infty$ at least as $\text{const.}\,x_{n}^{(k)}$.
However, since the energy decreases with each iteration, the sum of the other addends must diverge to $-\infty$, indeed we would have
$$\sum_{j,\ell=1,j\neq \ell}^nV_{x_j^{(k)}}(x_\ell^{(k)})\to-\infty\,.$$
But, since as $k\to\infty$
\begin{align*}
&\sum_{j,\ell=1,j\neq \ell}^nV_{x_j^{(k)}}(x_\ell^{(k)})>n(n-1)V_{-x_n^{(k)}}(x_n^{(k)})\\
=& n(n-1) \left(\frac{x_n^{(k)}}{h}\log\left(\frac{2x_n^{(k)}+h}{2x_n^{(k)}-h}\right)
-\frac{1}{2}\log((2x_n^{(k)}+h)(2x_n^{(k)}-h))\right)\\
=&n(n-1)\left(\frac{x_n^{(k)}}{h}\left(\frac{h}{x_n^{(k)}}+O\left(\left(x_n^{(k)}\right)^{-3}\right)
\right)-\frac{1}{2}\log((2x_n^{(k)}+h)(2x_n^{(k)}-h))\right)\\
&=-n(n-1)\log(x_n^{(k)})+O(1)
\end{align*}
this sum does not diverge sufficiently quickly to compensate for the behaviour of $\phi(x_n^{(k)})$. Hence, the sequence $\{(x_1^{(k)},\dots,x_n^{(k)})\}_{k\in\mathbb{N}}$ is bounded and we can apply the same arguments as in the bounded case.
\end{proof}

Method \ref{Alg} and theorem \ref{TeoMinEnergy} can be combined to obtain some properties of the equilibrium distribution. In the following results, we focus on the monotonicity and  interlacing properties.

\begin{cor}\label{CorMonoCeroshpos}
  Let $\phi_1$ and $\phi_2$ two G-convex external fields in intervals $(a_1,b_1)$ and $(a_2,b_2)$ (possibly unbounded) and $x_{1,1}<\dots<x_{n,1}$ and $x_{1,2}<\dots<x_{n,2}$ their equilibrium distributions  respectively.
  \begin{itemize}
  \item[(i)] If $a_1=a_2$, $b_1\leq b_2$ and $F_{\phi_1}< F_{\phi_2}$ in $(a_1,b_1)$ then
  $$x_{j,1}< x_{j,2}\,,\qquad j=1,\dots, n\,.$$
  \item[(ii)] If $a_1\leq a_2$, $b_1=b_2$ and $F_{\phi_1}> F_{\phi_2}$ in $(a_2,b_2)$ then
  $$x_{j,1}> x_{j,2}\,,\qquad j=1,\dots, n\,.$$
  \end{itemize}
%
\end{cor}

\begin{proof} We prove only \emph{(i)}, since the proof of \emph{(ii)} is analogous.
Note that when we apply the method \eqref{Alg}, if the total force over a charge at $x_j^{(k)}$
$$\sum_{\ell=1}^{j-1}F_{x_\ell^{(k+1)}}(x_j^{(k)})+\sum_{\ell=j+1}^{n}F_{x_\ell^{(k)}}(x_j^{(k)})
+F_\phi(x_j^{(k)})\,,$$
is positive, then $x_j^{(k+1)}>x_j^{(k)}$ (that is, the force is positive so the charge moves to the right). The consequence of this movement is that
$$F_{x_j^{(k+1)}}(x)>F_{x_j^{(k)}}(x)$$ for $x$ where the two forces are defined, so an additional positive force (pushing or pulling to the right) is added to every charge. Then, if we apply the method \ref{Alg} with the initial data $x_j^{(0)}=x_{j,1}$ for $j=1,\dots,n$ and the external field $\phi_2$, then the total force at each $x_j^{(0)}$ is $$\sum_{\ell=1}^{j-1}F_{x_\ell^{(0)}}(x_j^{(0)})+\sum_{\ell=j+1}^{n}F_{x_\ell^{(0)}}(x_j^{(0)})
+F_{\phi_2}(x_j^{(0)})=-F_{\phi_1}(x_j^{(0)})+F_{\phi_2}(x_j^{(0)})>0\,.$$
The consequences are that in any step, every charge moves to the right, then $\{x_j^{(k)}\}_{k\in\mathbb{N}}$ is an increasing sequence and by theorem \ref{TeoEDElechPos} converges to $x_{j,2}$ for any $j=1,\dots, n$. Then, the statement of the corollary holds.
%
%
%
\end{proof}

\begin{cor}\label{CorEntrelazamiento}
  Let $\phi$ a G-convex external field in an interval $(a,b)$ (possibly unbounded) and $x_{1,n}<\dots<x_{n,n}$ and $x_{1,n+1}<\dots<x_{n+1,n+1}$ the equilibrium distributions of $n$ and $n+1$ charges respectively. Then, these distributions interlace
  $$x_{1,n+1}<x_{1,n}<x_{2,n+1}<x_{2,n}<\dots <x_{n,n+1}<x_{n,n}<x_{n+1,n+1}\,.$$
\end{cor}
\begin{proof}
It is a direct consequence of corollary \ref{CorMonoCeroshpos}. The distribution $x_{1,n+1}<\dots<x_{n,n+1}$ is the equilibrium distribution in $(a,x_{n+1,n+1}-h)$ under the external field $\phi+V_{x_{n+1,n+1}}$, and its force
$$(F_\phi+F_{x_{n+1,n+1}})(x)<F_\phi(x)\,,\qquad x\in(a,x_{n+1,n+1}-h)\,.$$
Then $x_{j,n+1}<x_{j,n}$ for $j=1,\dots,n$.

In the same way, but by removing $x_{1,n+1}$, the distribution $x_{2,n+1}<\dots<x_{n+1,n+1}$ is the equilibrium distribution in $(x_{1,n+1}+h,b)$ under the external field $\phi+V_{x_{1,n+1}}$, and
$$(F_\phi+F_{x_{1,n+1}})(x)>F_\phi(x)\,,\qquad x\in(x_{1,n+1}+1,b)\,.$$
Then $x_{j,n}<x_{j+1,n+1}$ for $j=1,\dots,n$.
\end{proof}

Results in theorem \ref{TeoMinEnergy} and corollaries \ref{CorMonoCeroshpos} and \ref{CorEntrelazamiento} can be reinterpreted in terms of polynomial solutions of difference equations. Theorem \ref{TeoMinEnergy} says that if the polynomials $A$ and $B$ are such that $\phi$ is G-convex in an interval $(a,b)$, then there exists a unique polynomial $C_n$ such that there is a polynomial solution of the difference equation $A\Delta_h\nabla_h y+B \Delta_h y+C_ny=0\,,$ which is of degree $n$ and has simple roots in $(a,b)$ with the mutual distance between two consecutive roots greater than $h$. Corollary \ref{CorEntrelazamiento} also states the interlacing of the roots of these polynomial solutions for consecutive degrees. Finally corollary \ref{CorMonoCeroshpos} can be used to analyze the monotonicity of the roots of polynomial solutions of the same degree of these difference equations when the coefficients depends on some parameter. As a simple example of this, we formulate the following result.

\begin{cor}
  Let $A_\lambda$ and $B_\lambda$ two polynomials depending on the parameter $\lambda$ and such that there exists an interval $(a,b)$ ($a\in\mathbb{R}$ and $b\in (a+h,+\infty]$) satisfying the conditions:
  \begin{itemize}
  \item $1+B_\lambda/A_\lambda$ is positive and decreasing in $(a,b)$,
  \item $A_\lambda(a)=0\neq B_\lambda(a)$,
  \item If $b\in\mathbb{R}$ then $(A_\lambda+B_\lambda)(b)=0\neq A_\lambda(b)$ but if $b=+\infty$ then $\lim_{x\to+\infty}1+B(x)/A(x)\in[0,1)$.
  \end{itemize}
  Then, for each $n\leq\lceil(b-a)/h\rceil$ and $\lambda$ there exists a unique polynomial $C_{\lambda,n}$ such that the difference equation
  $$A_\lambda\Delta_h\nabla_h y+B_\lambda \Delta_h y+C_{\lambda,n}y=0\,,$$
  has a polynomial solution $p_{\lambda,n}$ of degree $n$ whose roots $x_{1,n,\lambda}<\dots<x_{n,n,\lambda}$ are in $(a,b)$ and $x_{j+1,n,\lambda}>x_{j,n,\lambda}+h$ for $j=1,\dots,n-1$.

  For a fixed $\lambda$ the roots of two consecutive polynomials $p_{\lambda,n+1}$ and $p_{\lambda,n}$ interlace
  $$x_{1,n+1,\lambda}<x_{1,n,\lambda}<x_{2,n+1,\lambda}<x_{2,n,\lambda}<\dots<x_{n-1,n,\lambda}
  <x_{n,n+1,\lambda}<x_{n,n,\lambda}<x_{n+1,n+1,\lambda}\,.$$

  If $B_\lambda(x)/A_\lambda(x)$ is strictly increasing in $\lambda$ for any $x\in(a,b)$, then
  $$x_{j,n,\lambda_1}<x_{j,n,\lambda_2}\,,\qquad $$
  for any $\lambda_1<\lambda_2$, $j=1,\dots,n,$ and $n\leq\lceil(b-a)/h\rceil$.
\end{cor}
\begin{remark} Observe that, in contrast to the Markov theorem for analyzing the monotonicity of these roots (see for instance \cite[Theorem 7.1.1]{IsmailBook2005} for a general version in the context of orthogonal polynomials), there is no need to use hypotheses on the continuity or differentiability of polynomials $A_\lambda$ and $B_\lambda$ with respect to $\lambda$. Furthermore, $\lambda$ could move in any subset (for instance a discrete one) of the real line.
\end{remark}

Finally, we study the extension of this problem to several disjoint intervals, although to simplify notation and ideas, we restrict ourselves to the cases of two intervals because cases with more intervals can be analyzed with the same definitions and results. Consider two disjoint intervals $(a,b)$ and $(c,d)$. We distribute charges in both intervals with the same forces as those given in \eqref{Forceh} (and with the same exclusion regions); this means that the potentials are also given by \eqref{Pothpos}. An external field of the type \eqref{ExtForcehPos} acts over both intervals, so we need the following assumption
$$1+\frac{B(x)}{A(x)}>0\,, \qquad x\in (a,b)\cup (c,d)\,.$$
The energy of the system is also given by \eqref{EnergiahPos}. A \emph{critical distribution} can be defined just as in definition \ref{DefCrDis} but we made a slightly different adaptation of the concept of \emph{equilibrium distribution}.
\begin{definition}
  Given $(n_1,n_2)\in\mathbb{N}^2$, we say that a distribution $x_1,\dots x_{n_1+n_2}$ is a $(n_1,n_2)$-equilibrium distribution under the external field $\phi$ if it corresponds to the solution of the minimizing problem
  $$\min\left\{\mathcal{E}_\phi(x_1,\dots,x_{n_1+n_2}): \begin{array}{l}x_1,\dots,x_{n_1}\in (a,b)\,,\, x_{n_1+1},\dots,x_{n_1+n_2}\in (c,d),\\
   |x_j-x_k|>h, j,k\in\{1,\dots,n_1+n_2\}, j\neq k\end{array}\right\}\,.$$
\end{definition}

Theorem \ref{TeoEDElechPos} is valid in this context too, thus there is an equivalence between critical distributions in two (or more) intervals and polynomial solutions of difference equations of type \eqref{ED}.

\begin{definition}\label{DefGconvex2}
  Given polynomials $A$ and $B$, such that $1+B/A>0$ in $(a,b)\cup (c,d)$, we say that the external field given by \eqref{ExtForcehPos} is G-convex in $(a,b)\cup (c,d)$ if its restrictions on $(a,b)$ and $(c,d)$ are both a G-convex external field in $(a,b)$ and $(c,d)$ respectively.
\end{definition}

Then, if the external field is G-convex, the method \ref{Alg} also works in this scenario and theorem \ref{TeoMinEnergy} is still valid (the proof is the same, observe that $K$ is a convex set, although $(a,b)\cup (c,d)$ does not, therefore the same arguments prove that the energy is a convex function). We only need the assumption that such amount of charges can be placed at $(a,b)$ and $(c,d)$ in such a way that the mutual distances are greater than $h$. We state the analogous to corollary \ref{CorMonoCeroshpos}  since the small details that appear.

\begin{cor}\label{CorMonoCeroshpos2}
Let $\phi_1$ and $\phi_2$ two G-convex external fields in the $(a_1,b_1)\cup(c_1,d_1)$ and $(a_2,b_2)\cup(c_2,b_2)$ (possibly unbounded) respectively and $x_{1,1}<\dots<x_{n_1+n_2,1}$ and $x_{1,2}<\dots<x_{n_1+n_2,2}$ their $(n_1,n_2)$-equilibrium distributions respectively.
  \begin{itemize}
  \item[(i)] If $a_1=a_2$, $c_1=c_2$, $b_1\leq b_2$ $d_1\leq d_2$ and $F_{\phi_1}< F_{\phi_2}$ in $(a_1,b_1)\cup(c_1,d_1)$ then
  $$x_{j,1}< x_{j,2}\,,\qquad j=1,\dots, n_1+n_2\,.$$
  \item[(ii)] If $a_1\leq a_2$, $b_1=b_2$, $c_1\leq c_2$, $d_1=d_2$ and $F_{\phi_1}> F_{\phi_2}$ in $(a_2,b_2)\cup(c_2,d_2)$ then
  $$x_{j,1}> x_{j,2}\,,\qquad j=1,\dots, n_1+n_2\,.$$
  \end{itemize}
\end{cor}

%

\section{Application to orthogonal polynomials in the Askey scheme}\label{Sec3}

If a polynomial $p_n$ of degree $n$ is orthogonal with respect to a positive discrete measure supported in $S=\{\xi_0, \xi_1,\dots \xi_N\}$ with $\xi_j=\xi_0+jh$ and $N\in\mathbb{N}\cup\{+\infty\}$, then its roots are known to be real and simple, furthermore, in each interval $(\xi_j,\xi_{j+1})$ there is at most one of them. The following lemma, which is a part of theorem 1 in \cite{KrZa09}, also guaranties that their mutual distances are always greater than $h$.

\begin{lm}\label{LemaOPhPos} If a polynomial $p_n(x)=(x-x_1)\dots(x-x_n)$ satisfies the following properties
\begin{itemize}
\item $p_n$ is orthogonal with respect to a positive measure, $\mu$, supported in $\Sigma=\{\xi_0, \xi_1,\dots \xi_N\}$ with $\xi_j=\xi_0+jh$ and $N\in\mathbb{N}\cup\{+\infty\}$ and $n<N+1$,
\item $p_n$ is a solution of the difference equation \eqref{ED} with $A$ and $B$ two polynomials such that $1+B/A>0$ in $(\xi_0,\xi_N)$,
\end{itemize}
then $|x_j-x_k|>h$ for all $j\neq k$.

As consequence the roots of $p_n$ correspond to a critical distribution in the external field $\phi$ such that
$$F_\phi(x)=\frac{1}{2h}\log\left(1+\frac{B(x)}{A(x)}\right)\,,\qquad x\in(\xi_0,\xi_N)\,.$$
\end{lm}

The following extension of Lemma \ref{LemaOPhPos} can be proved with the same arguments that appear in the proof of Theorem 1 in \cite{KrZa09} but adapted to the new scenario. It provides separation properties between the roots of orthogonal polynomials when they present some differential properties but in quadratic lattices.
  
\begin{lm}\label{LemaOPhPos2} Assume that $\lambda(x)=x(x+\delta)$ with $\delta\geq 1/2$, and $n<N\leq +\infty$. Let $p_n$ be a polynomial of degree $n$ such that:
\begin{itemize}
\item $p_n$ is orthogonal with respect to a positive measure, $\mu$, supported in $\Sigma=\{\lambda(\xi_0), \lambda(\xi_1),\dots \lambda(\xi_N)\}$,
    with $\xi_0\geq 0$ and $\xi_{j+1}=\xi_j+h$, $j=0,\dots N-1$.
\item $p_n(\lambda)$ is a solution of a difference equation of type \eqref{ED} with $A$ and $B$ two polynomials with $1+B/A>0$ in $(\xi_0,\xi_N)\cup(-\xi_N-\delta,-\xi_0-\delta)$.
\end{itemize}
then if $x_1,\dots,x_{2n}\in(\xi_0,\xi_N)\cup(-\xi_N-\delta,-\xi_0-\delta)$ are the $2n$ roots of $p_n(\lambda)$ then $|x_j-x_k|>h$ for all $j\neq k$.

As consequence the roots of $p_n(\lambda)$ correspond to a critical distribution in the external field $\phi$ such that
$$F_\phi(x)=\frac{1}{2h}\log\left(1+\frac{B(x)}{A(x)}\right)\,,\qquad x\in(\xi_0,\xi_N)\cup(-\xi_N-\delta,-\xi_0-\delta)\,.$$
\end{lm}

The Askey scheme (a part) is the natural context in which this electrostatic model can be applied as we see in the following subsections (see for instance \cite{KW} for data about hypergeometric representations, orthogonality and difference equations). When we deal with the uniform lattice ($\lambda(x)=x$), polynomials $A$ and $B$ have degree at most $2$ and $1$ respectively, hence $C$ is a constant. The weight function $w$ satisfies the Pearson-type equation
$$\Delta_h(A w)=Bw\,,\qquad \Rightarrow\qquad w(\xi_j)=\prod_{k=1}^{j}\frac{(A+B)(\xi_{j-1})}{A(\xi_j)}=\prod_{k=1}^{j}\frac{(A+B)(\xi_0+(j-1)h)}{A(\xi_0+jh)}\,,$$
and the polynomials solutions, $p$, of the difference equations \eqref{ED} satisfies the orthogonality relations
$$\sum_{j=0}^N p(\xi_j)\xi_j^k w(\xi_j)=0\,,\qquad k=0,\dots,\deg(p)-1\,,$$
where $\xi_0$ is a root of $A$ and $N$ is such that $\xi_N$ is a root of $A+B$ if it is possible or $N=+\infty$ in other case (assuming the convergence of the series).

\subsection{Charlier polynomials}
Charlier polynomials $C_n(x;a)$ are the polynomial solutions of the difference equations
$$x\Delta_1\nabla_1 y+(a-x)\Delta_1 y+n y=0\,,$$
so comparing with \eqref{ED} polynomials $A$ and $B$ are defined as
\begin{align*}
A(x)=x\,,\qquad B(x)=a-x\,.
\end{align*}
They can be represented explicitly via the hypergeometric function
$$C_n(x;a)=\hypergeometric{2}{0}{-n,-x}{-}{\frac{-1}{a}}
=\sum_{k=0}^n\frac{(-n)_k(-x)_k}{k!}\left(\frac{-1}{a}\right)^k$$
where $(a)_k=a(a+1)\dots(a+k-1)$ denotes the Pochhammer symbol. The weight function $w$ can be chosen as follows
$$w(x)=\frac{a^x}{\Gamma(x+1)}\,,$$
so the polynomials $\{C_n(x;a)\}_{n\in\mathbb{N}}$ are orthogonal with respect to the measure
$$\sum_{j=0}^\infty\frac{a^j}{j!}\delta_j\,,$$
where $\delta_x$ denotes the Dirac delta at the point $x$.  We consider only the case of a positive orthogonality measure, i.e. $a>0$, which implies that the roots are simple and located in $(0,+\infty)$. In virtue of lemma \ref{LemaOPhPos}, the roots  $x_{1,n}^{C}<x_{2,n}^{C}<\dots<x_{n,n}^{C}$ of $C_n(x;a)$ are such that $x_{j+1}^{C}-x_j^{C}>1$. Observe that just for the case $a>0$, we can consider that there is an external field $\phi^{C}$ with
$$F_{\phi^{C}}(x)=-(\phi^{C})'(x)
=\frac{1}{2}\log\left(1+\frac{B(x)}{A(x)}\right)
=\frac{1}{2}\log\left(\frac{a}{x}\right)\,,\qquad  x\in(0,+\infty)\,.$$
This is a G-convex external field so $\{x_{j,n}^{C}\}_{j=1}^n$ is the equilibrium distribution under $\phi^{C}$. Moreover, corollary \ref{CorMonoCeroshpos} states that $x_{j,n}^{C}$ is an increasing function of the parameter $a$ (which was already known, see for instance \cite{KCFRAS2020}).

%
%
%
%
%

\subsection{Krawtchouk polynomials}
Krawtchouk polynomial $K_n(x;p,N)$, $N\in\mathbb{N}$, is  the polynomial solution of the difference equation
$$A(x)\Delta_1\nabla_1 y+B(x)\Delta_1 y+n y=0\,,\qquad A(x)=(1-p)x\,,\quad B(x)=pN-x\,,$$
which implies that it has a hypergeometric representation as follows
$$K_n(x;p,N)=\hypergeometric{2}{1}{-n,-x}{-N}{\frac{1}{p}}
=\sum_{k=0}^n\frac{(-n)_k(-x)_k}{k!(-N)_k}\left(\frac{1}{p}\right)^k\,,\qquad n=0,1,\dots,N+1\,.$$
The weight function can be considered as
$$w(x)=\frac{\Gamma(N+1)}{\Gamma(x+1)\Gamma(N+1-x)}p^x(1-p)^{N-x}\,,$$
and hence the family $\{K_n(x;p,N): n=0,\dots,N+1\}$ is a sequence of orthogonal polynomials with respect to the measure
$$\sum_{j=0}^N\binom{N}{j}p^j(1-p)^{N-j}\delta_{j}\,,$$
which is a positive measure for $p\in(0,1)$, we only consider this case. Hence the roots of $K_n(x;p,N)$ are simple, contained in $(0,N)$ and their mutual distances are greater than $1$ (except for $n=N+1$, in which case the roots are $\{0,1,\dots,N\}$).

The external field $\phi^{K}$ is such that
$$F_\phi^{K}(x)=-(\phi^{K})'(x)=\frac{1}{2}\log\left(1+\frac{pN-x}{(1-p)x}\right)
=\frac{1}{2}\log\left(\frac{p(N-x)}{(1-p)x}\right)\,,$$
 thus, it is a G-convex external field in $(0,N)$. Therefore, the roots of $K_n(x;p,N)$, $x_{1,n}^{K}<x_{2,n}^{K}<\dots<x_{n,n}^{K}$ with $n\leq N$, correspond with the equilibrium distribution in $(0,N)$ under $\phi^{K}$. Since the quotient $(A+B)/A$ increases with $p$ and $N$, the roots of $x_{j,n}^{K}$ also increases as $p$ or $N$ increases (the monotonicity with respect to $p$ was previously known, see for instance \cite{KCMSFR2021}).

The roots of Krawtchouk and Charlier polynomials can be compared using corollary \ref{CorMonoCeroshpos}. If $a>Np/(1-p)$, then the  force of the external field associated with Charlier polynomials is greater than the force associated with the Krawtchouk case (in the sense that $F_{\phi^C}>F_{\phi^{K}}$) and hence, the roots of Charlier polynomials are greater than those of  Krawtchouk polynomials
\begin{equation}\label{CompCK}
x_{j,n}^{K}<x_{j,n}^{C}\,,\qquad j=1,\dots, n\,,\quad n\leq N\,.
\end{equation}

\subsection{Meixner polynomials} 

Meixner polynomials $M_n(x;\beta,c)$ are defined as the unique polynomial solution of the difference equation
$$A(x)\Delta_1\nabla_1 y+B(x)\Delta_1 y+n y=0\,,\qquad A(x)=x\,,\quad B(x)=c(x+\beta)-x\,,$$
which implies that the following hypergeometric representation holds
$$M_n(x;\beta,c)=\hypergeometric{2}{1}{-n,-x}{\beta}{1-\frac{1}{c}}
=\sum_{k=0}^\infty\frac{(-n)_k(-x)_k}{k!(\beta)_k}\left(1-\frac{1}{c}\right)^k\,.$$
The weight function is usually taken as
$$w(x)=\frac{\Gamma(\beta+x)c^x}{\Gamma(\beta)\Gamma(x+1)}\,,$$
and therefore, these polynomials are orthogonal with respect to the measure
$$\sum_{j=0}^\infty \frac{(\beta)_jc^j}{j!}\delta_j\,,$$
when $c\in(-1,1)$. This measure is positive when $\beta>0$ and $c\in(0,1)$, which are typically called the standard parameters. For these parameters, the roots of $M_n$ are simple, located in $(0,+\infty)$, and in virtue of lemma \ref{LemaOPhPos}, their mutual distances are greater than $1$.

The external field $\phi^M$ is such that
\begin{equation}\label{ExtFMeixner}
F_{\phi^{M}}(x)=-(\phi^{M})'(x)=\frac{1}{2}\log\left(\frac{c(x+\beta)}{x}\right)\,,\qquad x>0\,,
\end{equation}
so it can be considered as a constant force pushing to the left (with value $1/2\log(c)$) plus the force field created by a charge located at $-\beta/2$ of size $\beta/2$ and exclusion radius $\beta/2$. It is a G-convex external field in $(0,+\infty)$, so the roots $x_{1,n}^{M}<x_{2,n}^{M}<\dots<x_{n,n}^{M}$ of $M_n(x;\beta,c)$ are the equilibrium distribution under the external field $\phi^{M}$. Corollary \ref{CorMonoCeroshpos} proves that $x_{j,n}^{M}$ increases when the parameters $\beta$ or $c$ increase (the behaviour with respect to $\beta$ can be found for instance in \cite[p. 214]{IsmailBook2005} or in \cite{KCMSFR2021}).

When we compare the roots of Charlier $C_n(x;a)$ and Meixner polynomials $M_n(x;\beta,c)$ we see that when $a<c\beta$, we have $F_{\phi^{C}}<F_{\phi_M}$ in $(0,+\infty)$, so corollary \ref{CorMonoCeroshpos} implies
\begin{equation}\label{CompCM}
x_{j,n}^{C}<x_{j,n}^{M}\,,\qquad j=1,\dots,n\,,\quad n\in\mathbb{N}\,.
\end{equation}
Relations \eqref{CompCK} and \eqref{CompCM} can be combined to obtain
$$x_{j,n}^{K}<x_{j,n}^{C}<x_{j,n}^{M}\,,\qquad j=1,\dots,n\,,\quad n\leq N\,,$$
for the roots of $K_{n}(x;p,N)$ when $c\beta>Np(1-p)\,.$

Meixner polynomials can be considered also for some nonstandard parameters and in some of these cases, the results of section \ref{Sec2} can also be applied. This can be done in the following cases:
\begin{itemize}
\item Case $\beta>0$ and $c>1$. An orthogonality for Meixner polynomials with these parameters is not known; however, we can describe its roots. The external field \eqref{ExtFMeixner} is G-convex in $(-\infty,-\beta)$ and thus, for any $n>0$, there exists an unique equilibrium distribution in $(-\infty,-\beta)$, given by $x_1<x_2<\dots<x_n<-\beta$, with $x_{j+1}>x_j+1$. Using theorem \ref{TeoEDElechPos} we deduce that there is a polynomial $C$ of degree 0 such that $p(x)=(x-x_1)\dots(x-x_n)$ is a solution of the difference equation
$$A(x)\Delta_1\nabla_1 y(x)+B(x)\Delta_1 y(x)+C(x)y(x)=0\,.$$
By comparing the asymptotic term at $\infty$ in $x^{n}$ we deduce that $C(x)=n$, so indeed $p(x)=M_n(x;\beta,c)$. As a consequence, even with the absence of an orthogonality, we can deduce that the roots are located in $(-\infty,-\beta)$, that they are simple, that their mutual distances are greater than $1$ and that they are decreasing functions of $\beta$ and increasing  functions of $c$. Moreover, corollary \ref{CorEntrelazamiento} implies that the roots of $M_n(x;\beta,c)$ and $M_{n-1}(x;\beta,c)$ also interlace.

\item Case $c,\beta<0$. If $\beta$ is an integer, these polynomials coincide with the Krawtchouk polynomials $K_n(x;-\beta,c/(c-1))$, so the behaviour of its roots has been described in the previous section. For $\beta$ not an integer, an orthogonality for these parameters was obtained in \cite{CostasSLara2009}, but it is a non-hermitian complex orthogonality through a complex contour, so it does not explain how the roots behave. However, since the external field \eqref{ExtFMeixner} is G-convex in $(0,-\beta)$, the same arguments that in the previous case can be used, and hence, the roots of $M_n(x;\beta,c)$ are in $(0,-\beta)$, are simple and their mutual distances are greater than $1$.  Moreover, the roots of $M_n(x;\beta,c)$ and $M_{n-1}(x;\beta,c)$ also interlace and they decrease as $\beta$ or $c$ increases.
\end{itemize}

\subsection{Hahn polynomials}

Hahn polynomials $Q_n(x;\alpha,\beta,N)$ are the polynomial solutions of the difference equations $$A(x)\Delta_1\nabla_1 y+B(x)\Delta_1 y+n y=0\,,$$
with
$$A(x)=x(x-\beta-N-1)\,,\quad B(x)=(\alpha+\beta+2)x-(\alpha+1)N\,,$$
and $N\in\mathbb{N}$, so they have the hypergeometric representation
$$Q_n(x;\alpha,\beta,N)=\hypergeometric{3}{2}{-n,n+\alpha+\beta+1,-x}{\alpha+1,-N}{1}
=\sum_{k=0}^\infty\frac{(-n)_k(n+\alpha+\beta+1)_k(-x)_k}{k!(\alpha+1)_k(-N)_k}\,.$$
The support of the orthogonality measure is chosen as $\xi_j=j$ with $j=0,\dots,N$ and the weight function is
$$\rho(x)=\binom{\alpha+x}{x}\binom{\beta+N-x}{N-x}\,,$$
therefore $Q_n(x;\alpha,\beta,N)$, $n\leq N+1$, is orthogonal with respect to the measure
\begin{equation}\label{OrthMeasHahn}
\sum_{j=0}^N\binom{\alpha+j}{j}\binom{\beta+N-j}{N-j}\delta_j\,.
\end{equation}
It is a positive measure when  $\alpha,\beta\in(-1,+\infty)$ and also when $\alpha,\beta\in(-\infty,-N)$ (if we multiply it by $(-1)^N$), so in these cases, the roots of $Q_n$ are simple and belong to $(0,N)$. However, the quotient $(A+B)/A$ is positive in $(0,N)$ only for $\alpha,\beta\in(-1,+\infty)$ or $\alpha,\beta\in(-\infty,-N-1)$ so by direct application of lemma \ref{LemaOPhPos}, only in these cases we can guarantee that the mutual distances between the roots are greater than $1$.

In a first scenario we consider only the case $\alpha,\beta\in(-1,+\infty)$. The external field $\phi^{Q}$ is such that
$$F_{\phi^{Q}}(x)=-(\phi^{Q})'(x)=\frac{1}{2}\log\left(\frac{(x+\alpha+1)(N-x)}{x(N+\beta+1-x)}\right)\,.$$
This external field is created by two charges: one of them at $-(\alpha+1)/2$ with size $(\alpha+1)/2$ and exclusion radius $(\alpha+1)/2$ and the other at $(\beta+1)/2+N$ with size $(\beta+1)/2$ and exclusion radius $(\beta+1)/2$.
This is a G-convex external field in $(0,N)$, so the roots of $Q_n$, $x_{1,n}^{Q}<x_{2,n}^{Q}<\dots<x_{n,n}^Q$, correspond with the equilibrium distribution. 

Since $(A+B)/A$ on the interval $(0,N)$ is an increasing function in $\alpha$ and $N$ and a decreasing function in $\beta$, the roots $x_{j,n}^{Q}$ increases with $\alpha$ and $N$ and decreases with $\beta$ (the monotonicity with respect to $\alpha$ and $\beta$ can be found for instance in \cite[Theorem 7.1.2]{IsmailBook2005} or \cite{KCFRAS2020}).

If we compare its roots with those of Krawtchouk, we see that when $p/(1-p)>(\alpha+N+1)/(\beta+1)$ then $F_{\phi^Q}<F_{\phi^K}$ and then
$$x_{j,n}^{Q}<x_{j,n}^{K}\,,\qquad j=1,\dots,n\,,\quad n\leq N\,;$$
but when $p/(1-p)<(\alpha+1)/(N+\beta+1)$ then $F_{\phi^Q}>F_{\phi^K}$ and then
$$x_{j,n}^{Q}>x_{j,n}^{K}\,,\qquad j=1,\dots,n\,,\quad n\leq N\,.$$

In table \ref{TabHahn} all the possible scenarios in which $F_{\phi^Q}$ is a G-convex external field are listed, including the monotonicity of the roots with respect to the parameters. In cases with nonstandard parameters, the same arguments that we have seen for Meixner polynomials also hold, i.e., the external field is G-convex in the given interval and hence, for $n$ less than the ceil of length of the interval, there is a polynomial solution of a  difference equation
$$A\Delta_1\nabla_1 y+B\Delta_1 y+C y=0\,,$$
with $C$ some constant. Evaluating this equation in the polynomial solution and looking at the leading term at infinity, one can deduce that $C=n$. This proves that indeed, the polynomial solution of this difference equation is a Hahn polynomial with nonstandard parameters. Observe that even the case of $N$ a real number is considered, although it would be more natural to refer to these polynomials as Continuous Hahn polynomials (with complex parameters). The eight cases are deduced from the four combinations
$$r_1<p_1<r_2<p_2\,,\qquad r_1,r_2 \text{ roots of }\frac{A+B}{A}\qquad \text{and}\qquad p_1,p_2 \text{ poles of }\frac{A+B}{A}$$
together with the other four combinations 
$$p_1\leq p_2<r_1 \leq p_2\,,$$
and the fact that in all the combinations, the central interval ($(p_1,r_2)$ or $(p_2,r_1)$) is the only one where the external field can be G-convex.

\begin{table}
\begin{center}
\begin{tabular}{|c|c|c|c|}
\hline
 Case & Interval& Increasing& Decreasing\\
\hline
$N>0$,\quad $\alpha>-1$,\quad $\beta>-1$& $(0,N)$ & $\alpha$, \quad $N$& $\beta$\\
\hline
$N>0$,\quad $\alpha\leq -N-1$,\quad $\beta\leq -N-1$ & $(0,N)$ &$\beta$,\quad $N$ & $\alpha$\\
\hline
$N>0$,\quad $-N-1\leq \alpha<-1$, \quad $\beta\leq -N-1$& $(0,-\alpha-1)$ &$\beta$,\quad $N$ & $\alpha$\\
\hline
$N>0$,\quad  $\alpha \leq -N-1$,\quad $-N-1\leq\beta<-1$ & $(N+\beta+1,N)$ & $\beta$,\quad $N$ & $\alpha$\\
\hline
$N>0$,\quad $\alpha\geq -N-1$,\quad $\beta\geq-N-1$,\quad $\alpha+\beta<-N-2$ &  $(N+\beta+1,-\alpha-1)$ & $\beta$,\quad $N$ & $\alpha$ \\
\hline
$N<0$,\quad $\alpha< -1$,\quad $\alpha+\beta>-N-2$ & $(0,-\alpha-1)$ & - &$\alpha$,\quad $\beta$,\quad $N$\\
\hline
$N<0$,\quad  $\alpha+\beta>-N-2$,\quad $\beta<-1$ & $(N+\beta+1,N)$ & $\alpha$,\quad $\beta$,\quad $N$ & -\\
\hline 
$N<0$,\quad $\alpha>-1$,\quad $\beta>-1$,\quad $\alpha+\beta<-N-2$ & $(N+\beta+1,-\alpha-1)$ & $\beta$,\quad $N$ & $\alpha$ \\
\hline
\end{tabular}
\end{center}
\caption{All possible combinations of parameters for which the external field is G-convex. Third and fourth columns refer to the monotonicity of the roots of Hahn polynomials; if a parameter appears in the third one, it means that the roots are increasing functions of this parameter and analogously for the information in the fourth column.} \label{TabHahn}
\end{table}  

Incidentally, table \ref{TabHahn} gives some shaper bounds for the first or the last root of $Q_n(x;\alpha,\beta,N)$ where the orthogonality measure is positive, $\alpha,\beta\leq-N$ and  $n\leq\lceil\min\{N,-\alpha-1\}-\max\{0,N+\beta+1\}\rceil$. The roots of $Q_n$ were known to be in $(0,N)$ by the orthogonality, but if in addition $\alpha\in[-N-1,-N)$ then the greatest root of $Q_n$ is less than $-\alpha-1$ (see third and fifth rows in table \ref{TabHahn}). Similarly, if $\beta\in[-N-1,-N)$ then the least root of $Q_n$ is greater than $N+\beta+1$ (see fourth and fifth rows in table \ref{TabHahn}).

\subsection{Dual Hahn polynomials}

Dual Hahn polynomials $R_n(x;\gamma,\delta,N)$ with $N\in\mathbb{N}$ are defined in a similar way as the preceding ones but using the quadratic lattice $\lambda(x)=x(x+\gamma+\delta+1)$. This means that $R_n(\lambda(x);\gamma,\delta,N)$ is the polynomial (instead of $R_n(x;\gamma,\delta,N)$) which satisfies the differential properties expressed through the $\Delta_1$ and $\nabla_1$ operators. Then $R_n(\lambda(x);\gamma,\delta,N)$ is defined as the polynomial solution of
\begin{equation}\label{EDDualHahn}
A(x)\Delta_1\nabla_1 y(x)+B(x)\Delta_1 y(x)+C y(x)=0\,,
\end{equation}
with
\begin{align*}
&A(x)=x(x+\gamma+\delta+N+1)(x+\delta)(2x+\gamma+\delta+2)\,,\\
&A(x)+B(x)=(x+\gamma+1)(x+\gamma+\delta+1)(N-x)(2x+\gamma+\delta)\,,\\
&C(x)=n(2x+\gamma+\delta)(2x+\gamma+\delta+1)(2x+\gamma+\delta+2)\,.
\end{align*}
Hence it has the hypergeometric representation
$$R_n(\lambda(x);\gamma,\delta,N)=\hypergeometric{3}{2}{-n,-x,x+\gamma+\delta+1}{\gamma+1,-N}{1}\,.$$
The orthogonality measure is supported on the points $\xi_j=\lambda(j)$, $j=0,\dots,N$, thus when $n\leq N+1$, we have that $R_n(x;\gamma,\delta,N)$ is the orthogonal polynomial with respect to the measure
$$\sum_{j=0}^N\frac{(-1)^j(-N)_j (\gamma+1)_j(2j+\gamma+\delta+1)}{j!(j+\gamma+\delta+1)_{N+1}(\delta+1)_j}\delta_{\xi_j}\,.$$
This is a positive measure when $\gamma,\delta\in(-1,+\infty)$ and also when $\gamma,\delta\in(-\infty,-N)$. However, to clarify the explanations, we restrict ourselves to the case $\gamma,\delta\in[0,+\infty)$ in the remainder of this section, although other cases would be analyzed with the same tools. Therefore, the roots of $R_n(x;\gamma,\delta,N)$ are simple, located in $(0,\lambda(N))$ and in each interval $(\xi_j,\xi_{j+1})$, $j=1,\dots,N-1$ there is at most one of them. Then, the $2n$ roots of $R_n(\lambda(x);\gamma,\delta,N)$ are real, simple, and $n$ of them are in $(0,N)$ while the others $n$ are in $(-N-\gamma-\delta-1,-\gamma-\delta-1)$. Indeed, they are distributed symmetrically considering $(-\gamma-\delta-1)/2$ as the center of symmetry. We denote them by 
$$
0<x_1<x_2<\dots< x_n<N\,,\qquad \text{and}\qquad -\gamma-\delta-1-N<x_{n+1}<x_{n+2}<\dots<x_{2n}<-\gamma-\delta-1\,,$$
 with $x_{n+j}=-\gamma-\delta-1-x_{n-j+1}$. Hence, the roots of $R_n(x;\gamma,\delta,N)$ are $0<\lambda(x_1)<\lambda(x_2)<\dots<\lambda(x_n)<N$.

Observe that the quotient $(A+B)/A$ is positive in $(0,N)\cup(-\gamma-\delta-N-1,-\gamma-\delta-1)$ so we can consider the external field as
$\phi^R$ given by
$$F_{\phi^R}(x)=-(\phi^R)'(x)=\frac{1}{2}\log\frac{A(x)+B(x)}{A(x)}\,,\qquad x\in(0,N)\cup(-\gamma-\delta-N-1,-\gamma-\delta-1)\,.$$
Indeed, the identities 
\begin{align*}
&\frac{A+B}{A}\left(\frac{-\gamma-\delta-1}{2}+x\right)
=\left(\frac{A+B}{A}\left(\frac{-\gamma-\delta-1}{2}-x\right)\right)^{-1}\\
\Rightarrow\quad & F_{\phi^R}\left(\frac{-\gamma-\delta-1}{2}+x\right)=-  F_{\phi^R}\left(\frac{-\gamma-\delta-1}{2}-x\right)\\
\Rightarrow\quad& \phi^R\left(\frac{-\gamma-\delta-1}{2}+x\right)= \phi^R\left(\frac{-\gamma-\delta-1}{2}-x\right)\,,
\end{align*}
prove that it is a symmetric external field with respect to the center $(-\gamma-\delta-1)/2$.

Lemma \ref{LemaOPhPos2} can be used to deduce that $|x_j-x_k|>1$ for $j\neq k$, both in $\{1,\dots,2n\}$.

The external field $\phi^R$ is  G-convex (see definition \ref{DefGconvex2}) in the union of both intervals because
\begin{equation}
\begin{array}{ll}A(0)=0&A(-\gamma-\delta-N-1)=0\,,\\
 (A+B)(N)=0\qquad &(A+B)(-\gamma-\delta-1)=0
 \end{array}\end{equation}
and the rest of poles and roots of $(A+B)/A$ are outside of these intervals (the decrease of $F_{\phi^R}$ is not obvious, but it can be proved with standard techniques). Therefore, the distribution $x_1<\dots,<x_n$  together with $x_{n+1}<\dots<x_{2n}$ corresponds with the unique $(n,n)$-equilibrium distribution under the external field $\phi^R$ in $(0,N)\cup(-N-\gamma-\delta-1,-\gamma-\delta-1)$.

Let us see that although the same techniques for proving the monotonicity of the roots that we have used in the preceding sections are not applicable directly, we can make some modifications to analyze some of the problems of the movement of the zeros in this scenario. We use as an example the behaviour of the roots as $N$ increases.

The force that the external field exerts, $F_\phi(x)$, is an increasing function of $N$ for $x\in(0,N)$, which implies that it is decreasing for  $x\in(-N-\gamma-\delta-1,-\gamma-\delta-1)$. Therefore corollary \ref{CorMonoCeroshpos2} can not be used, but this problem can be solved by using a different version of method \ref{Alg}. In each step, instead of set one charge free, now we move a pair of symmetric charges to its equilibrium position. More specifically, in the first step we move the symmetric charges $x_1$ and $x_{2n}$ to the minimum of the problem
$$\min \{\mathcal{E}_\phi(x,x_2,\dots,x_{2n-1},y): x\in(0,x_2-1), y\in(x_{2n-1}+1,-\gamma-\delta-1)\}\,.$$
This problem has a unique solution since the external field
$$\sum_{j=2}^{2n-1}V_{x_j}+\phi\,,$$
is G-convex in $(0,x_2-1)\cup(x_{2n-1}+1,-\gamma-\delta-1)$ which implies that there an unique $(1,1)$-equilibrium distribution and then this movement is well defined.
In a second step, we move $x_2$ and $x_{2n-1}$ to the solution of the problem
$$\min \{\mathcal{E}_\phi(x_1,x,x_3,\dots,x_{2n-2},y,x_{2n}): x\in(x_1+1,x_3-1), y\in(x_{2n-2}+1,x_{2n}-1)\}\,.$$
The method follows in same way for the pair $x_3,x_{2n-2}$ and so on until the pair $x_n,x_{n+1}$. Then it begins again with the pair $x_1,x_n$ and so on.

The proof of theorem \ref{TeoMinEnergy} is also valid for this modified method, so we can guarantee that with any symmetric (with respect to $(-\gamma-\delta-1)/2$) distribution of $2n$ points in $(0,N)\cup(-N-\gamma-\delta-1,-\gamma-\delta-1)$ as initial data, the sequence of distributions given by the method converges to the $(n,n)$-equilibrium distribution.

Now we can prove that the positive roots of $R_n$ are increasing functions of $N$ by using the modified method. Let $\gamma,\delta\geq 0$ and $N_1<N_2$, $i=1,2$, and denote by $x_j^{i}$ the roots of $R_n(\lambda(x);\gamma,\delta,N_i)$ and by $\phi^i$ the corresponding external fields. We take $x_1^{1},\dots,x_{2n}^1$ as the initial data of the method applied with $\phi^2$ (in order to not complicate notation, we omit the superindex for the points in the steps of the method). In the first step, $x_1$ moves to the right while $x_{2n}$ moves to the left, which can be justified taking into account that $x_1$ must move to the unique root of the equation (see \eqref{TotalForcehPos})
$$\sum_{j=2}^{2n-1}\frac{1}{2}\log\left(\frac{x-x_j+1}{x-x_j-1}\right)
+\frac{1}{2}\log\left(\frac{2x+\gamma+\delta+2}{2x+\gamma+\delta}\right)+F_{\phi^2}(x)=0\,,$$
in $(0,x_2-1)$ (where the left hand side is decreasing function) and that, before the movement,
$$\sum_{j=2}^{2n-1}\frac{1}{2}\log\left(\frac{x_1-x_j+1}{x_1-x_j-1}\right)
+\frac{1}{2}\log\left(\frac{2x_1+\gamma+\delta+2}{2x_1+\gamma+\delta}\right)+F_{\phi^2}(x_1)=-F_{\phi^1}(x_1)+F_{\phi^2}(x_1)>0\,.$$
In the second step, the pair $x_2$ and $x_{2n-1}$ moves to its new equilibrium position, which means that $x_2$ moves to the right while $x_{2n-1}$ moves to the left. This is because all the changes in the forces that push to $x_2$ have increased except the one that exerts $x_{2n}$, but the effect of $x_{2n}$ on $x_2$ is weaker than the effect of $x_1$ on $x_2$ (the symmetry plays a crucial role in the argument), and the one that exerts $x_{2n-2}$ on $x_2$ which can be seen that it is not relevant (as we have seen in the first step with the movement of $x_{2n}$ and its effect on the movement of $x_1$). The same arguments let us guarantee that in all the steps of the method, the charges in the positive semiaxis move to the right and the ones in the negative semiaxis move to the left.

Finally we point out that, depending on the parameters, the external field can be G-convex in other intervals or union of intervals. Then as consequence of our results, there exists more polynomials $C$ such that the difference equation \eqref{EDDualHahn} has a polynomial solution (maybe different of the Dual Hahn polynomials) which roots are simple, located in these intervals and
with their mutual distances greater than 1. Furthermore, their behaviour can be described using corollaries \ref{CorMonoCeroshpos}, \ref{CorEntrelazamiento} or maybe \ref{CorMonoCeroshpos2}.

\subsection{Racah polynomials}


Racah polynomials $R_n(x;\alpha,\beta,\gamma,\delta)$ with $\alpha+1=-N$ or $\beta+\delta+1=-N$ or $\gamma+1=-N$ and $N\in\mathbb{N}$, are defined in a similar way to Dual Hahn polynomials. They, once we evaluate at the lattice $\lambda(x)=x(x+\gamma+\delta+1)$, are the unique polynomial solution of the difference equation
$$A(x)\Delta_1\nabla_1 y(x)+B(x)\Delta_1 y(x)+C(x) y(x)=0\,,$$
with
\begin{align*}
&A(x)=x(x-\beta+\gamma)(x+\delta)(x-\alpha+\gamma+\delta)(2x+\gamma+\delta+2)\,,\\
&A(x)+B(x)=(x+\alpha+1)(x+\gamma+1)(x+\beta+\delta+1)(x+\gamma+\delta+1)(2x+\gamma+\delta)\,,\\
&C(x)=n(2x+\gamma+\delta)(2x+\gamma+\delta+1)(2x+\gamma+\delta+2)\,.
\end{align*}
They has the hypergeometric representation
$$R_n(\lambda(x);\alpha,\beta,\gamma,\delta)
=\hypergeometric{4}{3}{-n,~n+\alpha+\beta+1,~-x,~x+\gamma+\delta+1}{\alpha+1,~\beta+\delta+1,~\gamma+1}{1}\,,$$
and are orthogonal with respect to the measure
$$\sum_{j=0}^N\frac{(\alpha+1)_j(\beta+\delta+1)_j(\gamma+1)_j(\gamma+\delta+1)_j((\gamma+\delta+3)/2)_j}
{j!(-\alpha+\gamma+\delta+1)_j(-\beta+\gamma+1)_j(\delta+1)_j((\gamma+\delta+1)/2)_j}\delta_j\,.$$
Among many possible choices of the parameters, we choose $\alpha=-N-1$, $\gamma,\delta\geq 0$ and $\beta>\gamma+N$, which guaranties that the orthogonality measure is positive. Hence,  $R_n(\lambda(x);\alpha,\beta,\gamma,\delta)$ has $n$ simple roots, $x_1<x_2<\dots<x_n$ in $(0,N)$ and others $n$, $x_{n+1}<x_{n+2}<\dots<x_{2n}$ in $(-N-\gamma-\delta-1,-\gamma-\delta-1)$ with their mutual distances greater than $1$ and symmetrically distributed with respect to the point $-(\gamma+\delta+1)/2$. The external field
$$\phi^R(x)=\frac{1}{2}\log\frac{A(x)+B(x)}{A(x)}\,,\qquad x\in(0,N)\cup(-N-\gamma-\delta-1,-\gamma-\delta-1)\,,$$
is G-convex since
$$A(0)=A(-N-\gamma-\delta-1)=0\,,\qquad (A+B)(N)=(A+B)(-\gamma-\delta-1)=0$$
and the rest of the roots or poles of $(A+B)/A$ are outside of $(0,N)\cup(-N-\gamma-\delta-1,-\gamma-\delta-1)$ (the decrease of $F_{\phi^R}$ can be proved with straightforward computations). Therefore, the points $x_1,\dots,x_{2n}$ correspond with the $(n,n)$-equilibrium distribution under the external field $\phi^R$.

With the choice of the parameters and with the same arguments that we have seen for Dual Hahn polynomials we can prove that the roots of the Racah polynomials are, for instance, increasing functions of $N=-\alpha-1\in\mathbb{N}$ and decreasing functions of $\beta$.

\section*{Acknowledgments}
The author is supported by Junta de Andalucía (research group FQM-384) and by the IMAG–Maria de Maeztu grant CEX2020-001105-M/AEI/10.13039/501100011033.

\bibliographystyle{amsplain}
\bibliography{DiscreteElectro}

\end{document}